\documentclass[10pt]{article}
\usepackage{amsmath,amsfonts,amsthm,amssymb,stmaryrd,mathrsfs}
\usepackage{enumerate, xspace, graphicx, url, microtype}
\usepackage[utf8]{inputenc}
\usepackage[T1]{fontenc}
\usepackage{hyperref}
\usepackage{tikz-cd}
\usepackage{rotating}
\newcommand*{\isoarrow}[1]{\arrow[#1,"\rotatebox{90}{\(\sim\)}"
]}
\usepackage[
    backend=biber,
    style=alphabetic,
    natbib=true,
    url=true, 
    doi=false,
    eprint=true
]{biblatex}
\newtheorem{theorem}{Theorem}[section]
\newtheorem{proposition}[theorem]{Proposition}
\newtheorem{corollary}[theorem]{Corollary}
\newtheorem{lemma}[theorem]{Lemma}

\newtheorem{theoremA}{Theorem}

\newtheorem{corollaryA}[theoremA]{Corollary}

\theoremstyle{definition}
\newtheorem{definition}[theorem]{Definition}

\newtheorem{remark}[theorem]{Remark}

\theoremstyle{remark}

\DeclareMathOperator{\GL}{GL} 
\DeclareMathOperator{\spec}{spec} 
\DeclareMathOperator{\Lie}{Lie} 
\DeclareMathOperator{\ad}{ad} 
 
\DeclareMathOperator{\DER}{{\mathcal D}er}
\DeclareMathOperator{\id}{id} 
\DeclareMathOperator{\pr}{pr}

\DeclareMathOperator{\At}{At}

\DeclareMathOperator{\Aut}{Aut}

\DeclareMathOperator{\B}{B} 
\DeclareMathOperator{\h}{H} 
\DeclareMathOperator{\car}{car} 
\DeclareMathOperator{\Z}{Z} 
\DeclareMathOperator{\C}{\mathscr C}

\newcommand{\dd}{\mathrm d}

\begin{document}
\title{Some consequences of \\ Frobenius descent for torsors}
\author{Niels Borne\thanks{Niels.Borne@univ-lille.fr} \and Mohamed Rafik Mammeri\thanks{mammeri.mohamed@gmail.com}}
\maketitle
\begin{abstract} We show how the formalism of Frobenius descent for torsors enables to study torsors under Frobenius kernels in terms of non-commutative, Lie-valued differential forms. We pay particular attention to affine line bundles trivialized by the Frobenius. \end{abstract}

\section{Introduction}%
\label{sec:Introduction}

Let $k$ be a field of positive characteristic $p$, that in this introduction we will assume perfect. Let $X/k$ be a smooth scheme, and denote by $F_X:X\to X$ the Frobenius morphism.

\subsection{The prototype}
\label{sub:prototype}
It is well known that the set of isomorphism classes of $\mathbb \mu_p$-torsors  on $X$ for the fppf topology is in one to one correspondence with the set of global $1$-differentials forms on $X$ that are closed and locally logarithmic.

It is fairly easy to give a rough idea of the correspondence: given such a form $\omega$, verifying locally $\omega =\dd f /f$, where $f$ is a unit, one defines a $\mu_p$-torsor locally by the Kummer equation $x^p=f$. Since two such $f$ differ by the $p$-th power of a \emph{unique} unit, the global form $\omega$ enables to glue the local torsors.

The actual proof relies on a delicate comparison of cohomologies on the étale and on the flat site, see \cite[Corollary 4.14]{milneEtaleCohomology1980}. 

There is a parallel characterization of   the set of isomorphism classes of $\mathbb \alpha_p$-torsors  on $X$ in terms of $1$-differentials forms on $X$ that are closed and locally exact. Moreover, these two cases have been generalized by Artin and Milne (see \cite{artin-milne:duality}) who give a description of $H^1(X,A)$ in terms of suitable Lie-valued differential forms if $A/k$ is an arbitrary abelian group scheme of height $1$.

\subsection{Frobenius descent enters the scene}%
\label{sub:Cartier descent enters the scene}

Our main point is that one can avoid the use of cohomology altogether and reinterpret the correspondence between $\mathbb \mu_p$-torsors and closed and locally logarithmic differential forms in terms of descent along the Frobenius. Namely, Cartier's Frobenius descent for vector bundles ensures that if $\mathcal E$ is a locally free sheaf on $X$, then $F_X^*\mathcal E$ is endowed with a canonical connection $\nabla_{can}$, whose moreover  both curvature $C$ and $p$-curvature $\psi$ vanish. And conversely, it is possible, in an essentially unique way, to descend such a vector bundle equipped with a suitable connection along $F_X:X\to X$.

If now $\omega$ is a closed $1$-form on $X$, and $\C$ is the Cartier operator, the condition $\C \omega = \omega$ not only characterizes the fact that $\omega$ is locally logarithmic, but also the fact that the connection $\mathrm{d}+\omega$ on the structure sheaf $\mathcal O_X$ is of vanishing $p$-curvature. By Frobenius descent, $(\mathcal O_X,\mathrm{d}+\omega)$ descends to an invertible sheaf $\mathcal L$ such that $F_X^*\mathcal L \simeq \mathcal O_X$. Since there is a natural isomorphism $F_X^*\mathcal L \simeq \mathcal L^{\otimes p}$, this gives rise to a $\mu_p$-torsor on $X$, as expected.

As Frobenius descent has been generalized from $G=\GL_n$ to $G$-torsors, where $G/k$ is a smooth affine group scheme (see \cite{chen-zhu:non_abelian, mammeriDescenteCartierTorseurs2016} and Theorem \ref{thm:cartier-descent-torsors}) the merits of this geometric approach are clear. 

\subsection{Main results}%
\label{sub:Main results}

To illustrate the usefulness of Frobenius descent for $G$-torsors, we work out completely the simplest non-commutative case, that is when the structure group is $G= \mathbb G_a \ltimes \mathbb G_m$:

\begin{corollaryA}[Corollary \ref{cor:Gm_rt_Ga_vanishing_p-curvature}]
	\label{cor:Gm_rt_Ga_vanishing_p-curvatureA}
Assume that $X/k$ is smooth, and that $k$ is perfect. Then the set of isomorphism classes of affine line bundles on $X$ trivialized along the Frobenius is in one to one correspondence with the set of pairs $(\omega,\omega')$ of $1$-forms on $X$ such that :
\begin{enumerate}
	\item $\mathrm{d} \omega =0 = \mathrm{d} \omega' +\omega\wedge \omega'$,
	\item $\C \omega = \omega$,
	\item for each Zariski open $U\subset X$ and each $f\in \mathbb G_m(U)$ such that $\omega_{|U}=\mathrm{d}f/f$, the equality $\C (f\omega'_{|U})=0$ holds.
\end{enumerate}
\end{corollaryA}
The first condition illustrates a new feature of the non-commutative situation, the appearance of non-closed forms.

Moreover, one gets a complete description of torsors under Frobenius kernels $G^F=\ker (F_G)$ in terms of suitable non-commutative Lie-valued differential forms:

\begin{theoremA}[Theorem \ref{thm:torsors_Frobenius_kernels}]
	\label{thm:torsors_Frobenius_kernelsB}
Assume that $X/k$ is smooth, and $k$ is perfect. There is a natural one to one map $\h^1(X,G^F)\to  \h^0(X,\Omega^1_{X/k}\otimes_k \Lie (G))^{C=\psi=0}$ that for each $g\in G(X)$ sends the image $\partial(g)$ of $g$ via the cobord map $\partial : G(X)\to \h^1(X,G^F)$ to $\dd \log(g)$, commutes with étale base change, and is characterized by these properties.
\end{theoremA}

\subsection{Some context}%
\label{sub:Some context}

The second result (Theorem \ref{thm:torsors_Frobenius_kernelsB}) appeared initially in the second author's phd (\cite{mammeriDescenteCartierTorseurs2016}). The exposition of its proof has been vastly improved.
 
The first result (Corollary \ref{cor:Gm_rt_Ga_vanishing_p-curvatureA}) also appeared in \cite{mammeriDescenteCartierTorseurs2016}, although in a more hidden form. The exposition chosen here is motivated by the recent surge of interest in affine bundles (see \cite{plechingerClassifyingAffineLine2019, plechingerEspacesModulesFibres2019, boucharebAffineHolomorphicBundles2024}). Affine line bundles are in fact a venerable subject, initiated by their classification on Riemann surfaces by André Weil (see \cite{zbMATH03116265}). 

Our study leads to the following conclusions. In the case where the structure group $G$ is abelian, Theorem \ref{thm:torsors_Frobenius_kernelsB} completely solves the issue of classifying $G$-torsors trivialized by the Frobenius, because in this situation there exists an explicit formula for the $p$-curvature (see Appendix \ref{sec:p-curvature_abelian_structure_groups}). However, Corollary \ref{cor:Gm_rt_Ga_vanishing_p-curvatureA} shows that even for the simplest non-abelian structure group, this classification is much more difficult. The case of a general non-abelian structure group remains mysterious.

\subsection{Structure of the text}%
\label{sub:Structure of the text}

In \S \ref{sec:atiyah_s_exact_sequence} we recall Atiyah's definition of a connection on a $G$-torsor, when $G/k$ is a smooth affine group scheme, and use it to define the Maurer-Cartan map $\rm d \log$. We also recall the definition of the associated curvatures.

In the next paragraph, \S \ref{sec:frobenius_descent_for_torsors}, we give an account of Frobenius descent for $G$-torsors, and show how it enables to describe the local structure of connections. We pay  particular attention to the case of special groups.

Paragraph \ref{sec:Affine line bundles trivialized along the Frobenius} is devoted to the study of affine line bundles. We apply the machinery of Frobenius descent to prove Corollary \ref{cor:Gm_rt_Ga_vanishing_p-curvatureA}.

Finally, in \S \ref{sec:Torsors under Frobenius kernels} we gather the previous results and prove Theorem \ref{thm:torsors_Frobenius_kernelsB}.

\subsection{Notations and conventions}
\label{sub:notations_and_conventions}

We work over an arbitrary field \( k \) of positive characteristic $p$ and denote by $S=\spec k$. Throughout $X$ is a smooth $k$-scheme. The tangent sheaf of $X$ is denoted by $\mathcal T_X$. By definition, this is the sheaf $\DER_k(\mathcal O_X)$ of derivations of $\mathcal O_X$ with respect to $k$.

We use $G$ for an affine $k$-group scheme. We consider mainly the fppf topology on $X$, and $G$-torsors $p:P\to X$  for this topology. We stick to EGA's convention and consider \emph{right} $G$-torsors. We denote as usual by $\B G(X)$ the category of $G$-torsors on $X$. We sometimes (especially when $G$ is smooth) use the \'etale topology on $X$; we then stick to the small \'etale site. 

If $\mathcal F$ is a quasi-coherent $G$-sheaf on $P$, we denote as usual by $p_*^G\mathcal F$ the quasi-coherent sheaf on $X$ obtained by pushing forward to $X$ and then taking fixed sections under the action of $G$. 
The adjoint action of $G$ on $\Lie(G)$ give rises to a $G$-sheaf $\mathcal O_P\otimes_k \Lie(G)$ on $P$, and we denote by $\ad(P)=p_*^G(\mathcal O_P\otimes_k \Lie(G))$ the associated locally free sheaf on $X$ (so $\ad(P)$ is the sheaf of sections of the adjoint vector bundle).

Finally, we denote by $F_{X/S}:X \rightarrow X^{(p)}$ the relative Frobenius of $X$ (see Appendix \ref{sec:classical_frobenius_descent} for details).

\subsection{Acknowledgements}
\label{sub:acknowledgements}

The first author wants to thank C\'edric P\'epin and Angelo Vistoli for many stimulating conversations. This work was supported in part by the Labex CEMPI (ANR-11-LABX-0007-01).

\section{Atiyah's exact sequence}
\label{sec:atiyah_s_exact_sequence}

Throughout this section, $G$ will denote a smooth affine group scheme over $k$.

\subsection{Invariant vector fields on a torsor}
\label{sub:invariant_vector_fields_on_a_torsor}

Atiyah's definition of a connection on a $G$-torsor relies on the following elementary fact:

\begin{lemma}
\label{lem:atiyah_s_exact_sequence}
Let $G/k$ be a smooth affine group scheme and $p:P\rightarrow X$ a $G$-torsor. Then there is a natural exact sequence :
\begin{equation}
\label{eq:atiyah_s_exact_sequence}
	0\rightarrow \ad(P) \rightarrow p_*^G(\mathcal T_P) \rightarrow  \mathcal T_X \rightarrow 0
\end{equation}
\end{lemma}

\begin{proof}
	\cite[Theorem 1]{atiyah:connections}.
\end{proof}

\begin{definition}[\cite{atiyah:connections}]
	\label{def:atiyah_connection}
An  Atiyah-connection on a $G$-torsor $p:P\rightarrow X$ is an \( \mathcal{O}_X \)-linear splitting $s:\mathcal T_X \to p_*^G(\mathcal T_P)$ of the exact sequence \eqref{eq:atiyah_s_exact_sequence}.
\end{definition}

We will use the notation $\At(G,X)$ for the category of $G$-torsors on $X$ endowed with a connection.

When $G=\GL_n$, one recovers the familiar notion of Koszul-connection, that is, a locally free sheaf $\mathcal{E}$ endowed with a $k$-linear map $\nabla: \mathcal{E} \rightarrow \mathcal{E}\otimes_{\mathcal{O}_X} \Omega^1_{X/k}$, satisfying Leibniz rule. 

It is clear that any section $\sigma:X\to P$ of a $G$-torsor $p:P\to X$ induces an Atiyah-connection $s_\sigma=p_*^G(\mathcal T_\sigma):\mathcal T_X\to p_*^G(\mathcal T_P)$. In particular, the unit section $e: X\to X\times G$ of the trivial torsor $\pr_1:X\times G\to X$ induces an Atiyah-connection $s_e$ on $X\times G$.

\begin{definition}[Maurer-Cartan map]
\label{def:Maurer-Cartan_map}
The Maurer-Cartan map 
$$\dd \log:G(X)\to \h^0(X,\Omega^1_{X/k}\otimes_{k}\Lie(G))\;\;,$$ denoted by $g \mapsto \Omega_g$, is the map given by, for all $g\in G(X)$: $\Omega_g=s_{g^{-1}}-s_e$.
\end{definition}

When $G=\GL_n$, there is an explicit and well-known formula for the Maurer-Cartan map.

\begin{lemma}
\label{lem:maurer-cartan-GL}
The map $\dd \log: \GL_n(X) \to \h^0(X,\Lie(\GL_n)\otimes_k \Omega_{X/k}^1)$ is given by: $\dd \log(g)= g^{-1}\mathrm{d}g  $.
\end{lemma}

\begin{proof}
	Given $g \in G$, the element $\dd \log(g)=\Omega_g$ of $\h^0(X,\Lie(\GL_n)\otimes_k \Omega_{X/k}^1)$ is characterized by the fact that left multiplication $L_g: (X\times G,s_e+\Omega_g) \to (X\times G,s_e)$ is an isomorphism of torsors endowed with Atiyah-connections. This translates into the fact that $g:(\mathcal O_X^{\otimes n},\mathrm{d}+\Omega_g ) \to (\mathcal O_X^{\otimes n},\mathrm{d})$ is a morphism of free sheaves endowed with Koszul-connections, which is equivalent to the displayed formula for $\dd \log(g)=\Omega_g$.
\end{proof}

By functoriality in $G$, we get the same formula for any smooth closed subgroup of $\GL_n$. For instance, by using the standard embedding $\mathbb G_a \to \GL_2$, $a\mapsto \begin{pmatrix}
	1 & a \\ 0&1
\end{pmatrix}$, we see that the Maurer-Cartan map for $\mathbb G_a$ is given by $a\mapsto \mathrm{d} a$.

For later use, we record the following lemma describing the image of $\dd \log$.

\begin{lemma}
\label{lem:image_of_MC}
Let $\Omega \in \h^0(X,\Omega^1_{X/k}\otimes_{k}\Lie(G))$. Then there exists $g\in G(X)$ such that $\Omega=\Omega_g$ if and only if $(X\times G,s_e+\Omega) \simeq (X\times G,s_e)$.
\end{lemma}

\begin{proof}
This is clear, since both assertions are equivalent to : there exists $g\in G(X)$ such that $s_e+\Omega=s_g$.
\end{proof}

\subsection{Curvatures}
\label{sub:curvatures}

\subsubsection{Usual curvature}%
\label{ssub:Usual curvature}

\begin{definition}
	\label{def:curvature-atiyah-connection}
Let $s:\mathcal T_X \to p_*^G(\mathcal T_P)$ be an Atiyah-connection on a $G$-torsor $p:P\rightarrow X$. Its curvature is the morphism $C_s:\mathcal T_X\times \mathcal{T}_X \to \ad(P)$ given for all $D,D'$ in \( \mathcal{O}_X \) by : 
\[ C_s(D, D')=[s(D),s(D')]-s([D,D']) \;\; .\]
  \end{definition}

  The curvature $C_s$ of an Atiyah-connection is a skew-symmetric $\mathcal O_X$-bilinear morphism (see \cite[\S 3.2]{langer:algebroid} ; in this language, the morphism $p_*^G \mathcal T_p: p_*^G(\mathcal T_P)\to \mathcal T_X$ can be seen as the anchor of a natural restricted Lie algebroid over $X/k$ with kernel $\left(p_*^G(\mathcal T_P)\right)^{(0)}=\ad(P)$ ).
We will also denote by $C_s:\bigwedge ^2 \mathcal{T}_X \to \ad(P)$ the corresponding $\mathcal O_X$-linear morphism. It can be seen as a global section of the sheaf $\Omega^2_{X/S}\otimes_{ \mathcal{O}_X} \ad (P)$. 

When $G=\GL_n$, it is easy to compute the curvature of connections on the trivial bundle:

\begin{lemma}
\label{lem:curvature-GL}
Let $\Omega \in \h^0(X,\Lie(\GL_n)\otimes_k \Omega_{X/k}^1)$. Then there is equality in $\h^0(X,\Lie(\GL_n)\otimes_k \Omega_{X/k}^2)$ : \( C_{s_e+\Omega}= \mathrm d \Omega + \Omega \wedge \Omega \).
\end{lemma}

\begin{proof}
	\cite[I (3.2.2)]{deligne:equadifs}. 
\end{proof}

\subsubsection{\texorpdfstring{$p$}{p}-curvature}%
\label{ssub:$p$-curvature}

\begin{definition}
\label{def:p-curvature-atiyah-connection}
Let $s:\mathcal T_X \to p_*^G(\mathcal T_P)$ be an Atiyah-connection on a $G$-torsor $p:P\rightarrow X$. The $p$-curvature $\psi_s:\mathcal{T}_X \to \ad P$ is the morphism  given for $D\in \mathcal T_X$ by
\[ \psi_s(D)=s(D)^p-s(D^p) \] 
\end{definition}

If the Atiyah connection $s$ is integrable, then the morphism $\psi_s$ is $p$-linear (see \cite[Lemma 4.9]{langer:algebroid}).

If $F_X$ denotes the absolute Frobenius of $X$, we will also write $\psi_s: \mathcal{T}_X \to {F_X}_*(\ad P)$ for the associated $\mathcal O_X$-linear morphism, or any of the equivalent following pieces of data : a $\mathcal O_X$-linear morphism $F_{X/S}^*\mathcal{T}_{X^{(p)}}\simeq {F_X}^* \mathcal{T}_X \to \ad P$ (where $F_{X/S}:X \rightarrow X^{(p)}$ is the relative Frobenius), or an element of  $\h^0(X,F_{X/S}^*\Omega^1_{X^{(p)}/k}\otimes_{\mathcal O_X} \ad(P))$.

In contrast to the usual curvature, when $G$ is non-abelian, there is no explicit formula (known to us) for the $p$-curvature of a connection on a trivial $G$-torsor, even when $G=\GL_n$. Nevertheless, we will be able to characterize, thanks to Cartier descent, the vanishing of $p$-curvature, see forthcoming Corollary \ref{cor:image_MC_torsors_vanishing_curvatures}. The first step is the following.
\begin{lemma}
	\label{lem:(p)-curvature_zero_etale_local}
Let $s:\mathcal T_X \to p_*^G(\mathcal T_P)$ be an Atiyah-connection on a $G$-torsor $p:P\rightarrow X$. The cancellation of $C_s$ is a local property with respect to the \'etale topology on $X$. If the connection $s$ is integrable, the same holds for the cancellation of $\psi_s$.
\end{lemma}

\begin{proof}
	Let $\mathcal F= \Omega^2_{X/S}\otimes_{ \mathcal{O}_X} \ad (P)$ (resp. $\mathcal F= F_X^*(\Omega^1_{X/S})\otimes_{ \mathcal{O}_X} \ad (P)$) . Since $\mathcal F$ is quasi-coherent on $X$, by flat descent, the pre-sheaf $(f:X'\to X)\mapsto \h^0(X',f^*\mathcal F)$ is a sheaf on the small \'etale site of $X$. Besides, since formation of $\Omega^1_{X/S}$  and of $\ad(P)$ commutes with étale base change, we have that $f^*\mathcal F=\Omega^2_{X'/S}\otimes_{ \mathcal{O}_{X'}} \ad (P')$ (resp. $f^*\mathcal F=F_{X'}^*(\Omega^1_{X'/S})\otimes_{ \mathcal{O}_{X'}} \ad (P')$), and the restriction morphism $\h^0(X,\mathcal F)\to \h^0(X',f^*\mathcal F)$ sends $C_s$ to ${C_{s'}}$ (resp. $\psi_s$ to ${\psi_{s'}}$), where $s'$ is the induced Atiyah-connection on the induced $G$-torsor $P'\to X'$.
\end{proof}

\section{Frobenius descent for torsors}%
\label{sec:frobenius_descent_for_torsors}

\subsection{Frobenius descent for torsors under a smooth group scheme}
\label{sub:frobenius_descent_for_torsors_under_a_smooth_group_scheme}

Let $G/k$ be a smooth affine group scheme. We first generalize the fact that if $\mathcal F$ is a locally free sheaf on $X^{(p)}$, then $F_{X/S}^* \mathcal{F}$ is endowed with a canonical connection $\nabla_{can}$ (see Appendix \ref{sec:classical_frobenius_descent}).

\begin{lemma}
\label{lem:canonical_connection}
The functor $F_{X/S}^*: \B G(X^{(p)})\rightarrow \B G(X)$ factors canonically through $\At(G,X)\rightarrow \B G(X)$.
\end{lemma}

\begin{proof}
	Let $p:P\to X^{(p)}$ be a $G$-torsor on $X^{(p)}$ and $q:X\times_{X^{(p)}}P \to X$ its pullback along $F_{X/S}:X\to X^{(p)}$. There is a natural morphism of exact sequences:
	\[ 
	\begin{tikzcd}
		0 \arrow[r] & \ad(X\times_{X^{(p)}}P) \arrow[r]\isoarrow{d} &  q_*^G(\mathcal T_{X\times_{X^{(p)}}P}) \arrow[r]\arrow[d,""]&  \mathcal T_X \arrow[r]\arrow[d,"0"] & 0 \\
		0 \arrow[r] & F_{X/S}^*\ad(P) \arrow[r] &  F_{X/S}^*{p}_*^{G}(\mathcal T_{P}) \arrow[r]&  F_{X/S}^*\mathcal T_{X^{(p)}} \arrow[r] & 0
\end{tikzcd}
\] 
Since the Frobenius morphism induces $0$ on tangent spaces, the morphism  $q_*^G(\mathcal T_{X\times_{X^{(p)}}P})\to F_{X/S}^*{p}_*^{G}(\mathcal T_{P})$ factors through $\ad(X\times_{X^{(p)}}P) \to F_{X/S}^*{p}_*^{G}(\mathcal T_{P})$, and this gives a retraction of the top exact sequence, or equivalently a section $s_{can}: \mathcal T_X \to  q_*^G(\mathcal T_{X\times_{X^{(p)}}P})$. Thus we get a canonical Atiyah-connection (Definition \ref{def:atiyah_connection}) $s_{can}$ on  $X\times_{X^{(p)}}P$.
\end{proof}

We use the symbol $(F_{X/S}^*\cdot, s_{can}) : \B G(X^{(p)})\rightarrow \At(G,X)$ (or simply $F_{X/S}^*$ when there is no ambiguity) to denote the functor induced by Lemma \ref{lem:canonical_connection}.

\label{sec:Frobenius descent for torsors}
\begin{theorem}[Frobenius descent for $G$-torsors]
\label{thm:cartier-descent-torsors}

Let $X/k$ be a smooth scheme and $G/k$ a smooth affine group scheme. The functor $(F_{X/S}^*\cdot, s_{can})$  induces an equivalence of categories: 
\[ (F_{X/S}^*\cdot, s_{can}) : \B G(X^{(p)})\rightarrow \At(G,X)^{C=\psi=0} \]
between the category of $G$-torsors on $X^{(p)}$ and the category of $G$-torsors on $X$ endowed with an Atiyah-connection with vanishing curvature and vanishing $p$-curvature.
\end{theorem}

The aim of this note is not to discuss a proof of this theorem, but to give some applications. For a tannakian proof, see \cite{mammeriDescenteCartierTorseurs2016}. For a sketch of a more direct proof, using connexions on the algebras of the torsors, see \cite{chen-zhu:non_abelian}. Both proofs ultimately rely on the $G=\GL_n$ case.

\subsection{Torsors with vanishing curvatures}%
\label{sub:Torsors with vanishing curvatures}

We give a first application of Theorem \ref{thm:cartier-descent-torsors} to the local structure of connections on torsors with vanishing curvature and vanishing $p$-curvature.

\begin{proposition}
\label{prop:local_structure_torsors_vanishing_curvatures}
With the hypothesis of Theorem \ref{thm:cartier-descent-torsors}, an Atiyah-connection $s$ on a $G$-torsor $p:P\to X$ verifies $C_s=\psi_s=0$ if and only if étale locally on $X$, there exists an isomorphism $(P,s)\simeq (X\times G,s_e)$.
\end{proposition}

\begin{proof}
	The ``if'' part is clear, since the cancellation of the ($p$-)curvature is a property that is étale-local (see Lemma \ref{lem:(p)-curvature_zero_etale_local}). For the ``only if'' direction: if we assume that $C_s=\psi_s=0$, then by Theorem \ref{thm:cartier-descent-torsors}, $P\to X$ descends along $F_{X/S}:X\to X^{(p)}$ into a $G$-torsor $\widetilde{P}\to X^{(p)}$. Since $G$ is smooth, there exists an étale covering $(X^{(p)}_i\to X^{(p)} )_{i\in I}$ such that the $G$-torsor $\widetilde{P}\to X^{(p)}$ is trivial when restricted along $X^{(p)}_i\to X^{(p)}$, for each $i\in I$. By applying the functor $(F_{X/S}^*\cdot, s_{can})$, we see that $(P,s)\simeq (G_X,s_e)$ when restricted along each arrow of the pull-back étale cover $(X_i\to X)_{i\in I}$ of $X$.
\end{proof}

\begin{corollary}
	\label{cor:image_MC_torsors_vanishing_curvatures}
Let $\Omega \in \h^0(X,\Omega^1_{X/k}\otimes_{k}\Lie(G))$. Then the Atiyah-connection $s_e+\Omega$ on $X\times G\to G$ has vanishing curvature and $p$-curvature if and only if étale locally on $X$, $\Omega$ is in the image of the Maurer-Cartan map.
\end{corollary}

\begin{proof}
	This follows from Proposition \ref{prop:local_structure_torsors_vanishing_curvatures} and Lemma \ref{lem:image_of_MC}.
\end{proof}

\subsection{Special groups}
\label{sub:special_groups}

Let $G/k$ be a group scheme, and $X$ a scheme. We recall that a fpqc $G$-torsor $P\to X$ is called \emph{locally isotrivial} if for any point $x\in X$, there exists a Zariski open $U\subset X$ containing $x$, and a finite étale cover $U'\to U$, such that $U'\times_X P \to U'$ is trivial as a $G$-torsor. In fact, when $G/k$ is a smooth affine group scheme, any fpqc torsor is locally isotrivial (see \cite[Lemme XIV.1.4]{raynaud:faisceaux_amples}).

Serre in \cite[\S 4.4]{serre:espaces_fibres} (see also \cite[\S 3]{grothendieck:torsion}) defines a \emph{special group} as a group scheme $G/k$ such that any fpqc $G$-torsor that is locally isotrivial is already locally trivial in the Zariski topology (the definition is given for an algebraically closed field but makes sense over an arbitrary field). For a modern treatment, see \cite{reichsteinSpecialGroupsVersality2020}. Classical examples include $\mathbb G_m$ (more generally $\GL_n$) and $\mathbb G_a$.

If $G/k$ is a smooth affine group scheme, that is special in the above sense, then 
Proposition \ref{prop:local_structure_torsors_vanishing_curvatures} and
Corollary \ref{cor:image_MC_torsors_vanishing_curvatures} hold verbatim if we replace the word `\'etale' by the word `Zariski'.

\subsection{Torsors trivialized along the relative Frobenius}%
\label{sub:Torsors trivialized along the relative Frobenius}

\begin{definition}[]
\label{def:torsors_trivialized_relative_Frobenius}
We define the category $\B  G(X^{(p)})^{F_{X/S}-triv}$ as the $2$-product fitting in the following diagram~:

\[ 
\begin{tikzcd}
	\B  G(X^{(p)})\ar[r,"F_{X/S}^*"] & \B G (X)\\
	\B  G(X^{(p)})^{F_{X/S}-triv}      \ar[r]\ar[u,""] & * \ar[u,"X\times G\to X "']
\end{tikzcd}
\]
where $*$ denotes the discrete category with a single object.
 
\end{definition}

\subsection{Atiyah connections on trivialized bundles}%
\label{sub:Atiyah connections on trivialized bundles}

\begin{definition}
\label{def:atiyah_connection_infinitesimal}
 Similarly, we define the category $\At^{triv}(G,X)$ as the $2$-product fitting in the following diagram :

\[ 
\begin{tikzcd}
	\At(G,X)\ar[r] & \B G(X)\\
\At^{triv}(G,X)\ar[r]\ar[u,""] & * \ar[u,"X\times G\to X "']
\end{tikzcd}
\]

\end{definition}

It is clear that $\At^{triv}(G,X)$ is a discrete category, more precisely: the functor
\begin{equation}
\label{eq:discrete-equivalence}
s_e+ \cdot: \h^0(X,\Omega^1_{X/k}\otimes_{k}\Lie(G)) \to \At^{triv}(G,X)
\end{equation}
sending $\Omega$ to $((X\times G\to X,s_e+\Omega),\id)$ is an equivalence of discrete categories, where we see the set on the left as a category with only identities as morphisms. There is even an explicit inverse equivalence given by transfer of structure.

We denote by $\At^{triv}(G,X)^{C=\psi=0}$ the full subcategory of $\At^{triv}(G,X)$ consisting of connections with vanishing curvatures, and similarly $\h^0(X,\Omega^1_{X/k}\otimes_{k}\Lie(G))^{C=\psi=0}$ for the corresponding Lie-valued $1$-forms. It follows from Theorem \ref{thm:cartier-descent-torsors} that the functor $(F_{X/S}^*\cdot, s_{can})$  induces an equivalence of categories: 
\[ (F_{X/S}^*\cdot, s_{can}) : \B G(X^{(p)})^{F_{X/S}-triv} \rightarrow \At^{triv}(G,X)^{C=\psi=0} \]

\subsection{Torsors trivialized along the absolute Frobenius}%
\label{sub:Torsors trivialized along the absolute Frobenius}

\begin{definition}[]
\label{def:torsors_trivialized_absolute_Frobenius}
We define the category $\B  G(X)^{F_{X}-triv}$ as the $2$-product fitting in the following diagram~:

\[ 
\begin{tikzcd}
	\B  G(X)\ar[r,"F_{X}^*"] & \B G^{(p)} (X)\\
	\B  G(X)^{F_{X}-triv}      \ar[r]\ar[u,""] & * \ar[u,"X\times G^{(p)}\to X "']
\end{tikzcd}
\]

\end{definition}

We will see a natural interpretation of this definition in section \ref{sec:Torsors under Frobenius kernels} but first use it in section \ref{sec:Affine line bundles trivialized along the Frobenius}. When the base field is perfect, it is clear that the functor $\B  G(X)^{F_{X}-triv} \rightarrow \B G^{(p)}(X^{(p)})^{F_{X/S}-triv}$ is an equivalence, so that we get a natural equivalence $\B  G(X)^{F_{X}-triv} \rightarrow \h^0(X,\Omega^1_{X/k}\otimes_{k}\Lie(G^{(p)}))^{C=\psi=0}$.

\section{Affine line bundles trivialized along the Frobenius}%
\label{sec:Affine line bundles trivialized along the Frobenius}

\subsection{Affine line bundles}%
\label{sub:Affine line bundles}

\begin{definition}[]
\label{def:affine_line_bundle}
An affine line bundle over a scheme $S$ is a couple $(L,A)$ where $L$ is a line bundle and $A$ is a torsor under $L$.
\end{definition}

The morphisms are the pairs $(\phi,\psi): (L,A)\to (L',A')$ so that $\phi :L\to L'$ is linear and $\psi :A\to A'$ is $\phi$-equivariant.
 
Since each object is Zariski locally isomorphic to the trivial object $(\mathbb G_a, \mathbb A^1)$, affine line bundles form a neutral gerbe banded by $\Aut(\mathbb G_a, \mathbb A^1)=\mathbb G_a \ltimes \mathbb G_m$. This \emph{affine group} will be denoted by $aff(1)$.

\begin{remark}
\label{rem:affine_line_bundle}

V. Plechinger gives another description of this category as the category of epimorphisms $\mathcal E\twoheadrightarrow \mathcal O_S$ where $\mathcal E$ is a rank $2$ vector bundle on $S$, see \cite{plechingerEspacesModulesFibres2019, plechingerClassifyingAffineLine2019}. It is easy to work out an explicit equivalence.
 
\end{remark}

\subsection{de Rham interpretation}%
\label{sub:de Rham interpretation}

We are now back to our main setup where $X/k$ is a smooth variety.
After a brief reminder of the cases $G=\mathbb G_m$ and $G=\mathbb G_a$, we give an interpretation of affine line bundles (trivialized by the Frobenius) in terms of differential forms.

\subsubsection{Cartier operator}%
\label{ssub:Cartier operator}

We denote by $\C: \h^0(X, \Z(\Omega^1_{X/k})) \rightarrow \h^0(X^{(p)},\Omega^1_{X^{(p)}/k})$ the Cartier operator. For its definition, we refer for instance to L.Illusie's article in \cite{bertinIntroductionTheorieHodge1996a}, in particular, to \S 3 therein. More generally, the Cartier operator defines an isomorphism

$$ \C: \oplus_{i\geq 0}\mathcal H^i\left( (F_{X/k})_* \Omega^{\bullet}_{X/k}\right) \rightarrow \oplus_{i\geq 0}\Omega^i_{X^{(p)}/k} \;\; .	$$

	 \subsubsection{Line bundles and \texorpdfstring{$1$}{1}-forms}%
\label{ssub:line_bundles_and_1_forms}

\begin{proposition}[{\cite[Corollary 7.1.3]{katzAlgebraicSolutionsDifferential1972}}]
\label{prop:line_bundles_and_1_forms}

Let $\omega \in \h^0(X, \Z(\Omega^1_{X/k}))$ be a closed $1$-form. The following statements are equivalent:
\begin{enumerate}
	\item  \label{Gmfirst} $\psi(\mathcal O_X, \mathrm{d}+\omega)=0$,
	\item \label{Gmsecond} $\C \omega = \pi_{X/k}^*\omega$,
	\item \label{Gmthird} $\omega$ is locally logarithmic.
\end{enumerate}
\end{proposition}

\begin{remark}
\label{rem:line_bundles_and_1_forms}
Here $\pi_{X/k}:X^{(p)}\rightarrow X$ denotes the canonical morphism (see Appendix \ref{sec:classical_frobenius_descent} for details). When $X$ is perfect, we identify the isomorphism $\pi_{X/k}$ with $\id_X$ and write the above equality as $\C \omega = \omega$. 
\end{remark}

\begin{proof}
	The assertions \ref{Gmfirst} and \ref{Gmthird} are equivalent by Corollary \ref{cor:image_MC_torsors_vanishing_curvatures}. The equivalence of \ref{Gmfirst} and \ref{Gmsecond} follows from a direct computation, see \cite[Proposition 7.1.2]{katzAlgebraicSolutionsDifferential1972}   
\end{proof}

\subsubsection{\texorpdfstring{$\mathbb G_a$}{Ga}-torsors and \texorpdfstring{$1$}{1}-forms}%
\label{ssub:G_a-torsors_and_1-forms}

Since $\Lie(\mathbb G_a)=k$, any $1$-form $\omega$ also defines a connection on the trivial $\mathbb G_a$-torsor. The analogue of Proposition \ref{prop:line_bundles_and_1_forms} is well known:

\begin{proposition}
\label{prop:G_a-torsors_and_1_forms}

Let $\omega \in \h^0(X, \Z(\Omega^1_{X/k}))$ be a closed $1$-form. The following statements are equivalent:
\begin{enumerate}
	\item  \label{Gafirst} $\psi(\mathbb G_a, \mathrm{d}+\omega)=0$,
	\item \label{Gasecond} $\C \omega = 0$,
	\item \label{Gathird} $\omega$ is locally exact.
\end{enumerate}
\end{proposition}

\begin{proof}
	Again, the assertions \ref{Gafirst} and \ref{Gathird} are equivalent by Corollary \ref{cor:image_MC_torsors_vanishing_curvatures}. The equivalence of \ref{Gasecond} and \ref{Gathird} follows from Cartier's theorem (\cite[Théorème 3.5,p.130]{bertinIntroductionTheorieHodge1996a}). 
\end{proof}

\subsubsection{de Rham interpretation of affine line bundles}%
\label{ssub:de_Rham_interpretation_of_affine_line_bundles}

For the rest of this section we set $G=aff(1)$. This is a special group, so according to \S \ref{sub:special_groups} we can work with the Zariski topology rather than the étale topology. 
 
 We use the standard embedding $G\to \GL_2$ given by $(a,b)\mapsto
 \begin{pmatrix}
	 a & b\\ 0 &1
 \end{pmatrix}$
 . Thus $\Lie{G}$ identifies with $\left \{\begin{pmatrix}
	 a & b\\ 0 &0
 \end{pmatrix}, a,b \in k \right\}$, hence connections on the trivial $G$-torsor are given by matrices $$\Omega=\begin{pmatrix}
	 \omega & \omega'\\ 0 &0
 \end{pmatrix}$$
 where $\omega, \omega'$ are two $1$-forms on $X$. Lemma \ref{lem:curvature-GL} gives immediately :
$$C_{s_e+\Omega}=\begin{pmatrix}
	\mathrm{d} \omega & \mathrm{d} \omega' +\omega\wedge \omega'\\ 0 &0
 \end{pmatrix}$$

\begin{proposition}
\label{prop:Gm_rt_Ga_vanishing_p-curvature}
With notations as above, assume that $\mathrm{d} \omega =0 = \mathrm{d} \omega' +\omega\wedge \omega'$. Then $\psi_{s_e+\Omega}=0$ if and only if $\C \omega = \pi_{X/k}^*\omega$ and for each Zariski open $U\subset X$ and each $f\in \mathbb G_m(U)$ such that $\omega_{|U}=\mathrm{d}f/f$, the equality $\C (f\omega'_{|U})=0$ holds.
\end{proposition}

\begin{remark}
\label{rem:Gm_rt_Ga_vanishing_p-curvature}
	The condition $\C (f\omega'_{|U})=0$ makes sense because $\mathrm{d} \omega' +\omega\wedge \omega'=0$ ensures that $f\omega'_{|U}$ is closed.
 
\end{remark}

\begin{proof}[Proof of Proposition \ref{prop:Gm_rt_Ga_vanishing_p-curvature}]
	Frobenius descent, and in particular Corollary \ref{cor:image_MC_torsors_vanishing_curvatures}, implies that $\psi_{s_e+\Omega}=0$ if and only if $\Omega$ is Zariski locally in the image of the Maurer-Cartan map. If 
	$g=
	\begin{pmatrix}
		f & f' \\ 0 & 1	
	\end{pmatrix};$
 then $g^{-1} \mathrm{d}g =
\begin{pmatrix}
		f^{-1} \mathrm{d}f & -f^{-1}\mathrm{d}f' \\  & 0	
\end{pmatrix}$, so $\psi_{s_e+\Omega}=0$ if and only if Zariski locally there exists $f\in \mathbb G_m$ and $f'\in \mathbb G_a$ such that $\omega = f^{-1} \mathrm{d}f$ and $\omega'= -f^{-1}\mathrm{d}f' $. Hence the `if' direction follows from the cases $G=\mathbb G_m$ and $G=\mathbb G_a$ (Propositions \ref{prop:line_bundles_and_1_forms} and \ref{prop:G_a-torsors_and_1_forms}). For the  `only if' direction : it follows from the $G=\mathbb G_m$ case that $\C \omega = \pi_{X/k}^*\omega$. We now fix a Zariski open $U\subset X$ and $f\in \mathbb G_m(U)$ such that $\omega_{|U}=\mathrm{d}f/f$. By hypothesis, there exists a Zariski cover $(U_i)_{i\in I}$ of $X$, and for each $i\in I$, a pair $(f_i,f'_i)$ in $G(U_i)$ such that $\omega_{|U_i} = f_i^{-1} \mathrm{d}f_i$ and $\omega'_{|U_i}= -f_i^{-1}\mathrm{d}f'_i $. On each $U\cap U_i$, since we assume that $X/k$ is smooth, there exists $f''_i$ in  $\mathbb G_m((U\cap U_i)^{(p)})$ such that $f_{|U\cap U_i}=  F_{U\cap U_i/S}^*(f''_i)  {f_i}_{|U\cap U_i}$
Hence by $p^{-1}$-linearity of the Cartier operator, for each $i\in I$ we have : \( \C(f_{|U\cap U_i} \omega' _{|U\cap U_i} )= f''_i\C({f_i}_{|U\cap U_i} \omega' _{|U\cap U_i} )= 0\), which implies $\C (f\omega'_{|U})=0$.

\end{proof}

We get the following quite complete description of the set of affine line bundles on $X$ trivialized along the absolute Frobenius :

\begin{corollary}
\label{cor:Gm_rt_Ga_vanishing_p-curvature}
Assume that $X/k$ is smooth, and that $k$ is perfect. Then the set of isomorphism classes of affine line bundles on $X$ trivialized along the Frobenius is in one to one correspondence with the set of pairs $(\omega,\omega')$ of $1$-forms on $X$ such that :
\begin{enumerate}
	\item $\mathrm{d} \omega =0 = \mathrm{d} \omega' +\omega\wedge \omega'$,
	\item $\C \omega = \omega$,
	\item for each Zariski open $U\subset X$ and each $f\in \mathbb G_m(U)$ such that $\omega_{|U}=\mathrm{d}f/f$, the equality $\C (f\omega'_{|U})=0$ holds.
\end{enumerate}
\end{corollary}

\begin{proof}
	This follows from section \ref{sub:Torsors trivialized along the absolute Frobenius} and Proposition \ref{prop:Gm_rt_Ga_vanishing_p-curvature}.
\end{proof}

\section{Torsors under Frobenius kernels}%
\label{sec:Torsors under Frobenius kernels}

\subsection{Frobenius kernels}%
\label{sub:Frobenius kernels}

Let $k$ be a field and $G/k$ be a smooth affine group scheme. We define its Frobenius kernel as $G^F:=\ker(G\to G^{(p)})$. In this \S \ref{sub:Frobenius kernels}, we will give a complete description of $G^F$-torsors on a smooth scheme $X/k$, when $k$ is perfect.

Lemma \ref{lem:frobenius_pull_back_of_torsors} implies that the functor $F_{X}^*:\B  G(X) \rightarrow \B  G^{(p)}(X)   $ is isomorphic with restriction of the structure group along $G\to G^{(p)}$. It follows from Definition \ref{def:torsors_trivialized_absolute_Frobenius} that there is a natural equivalence $ \B G^F(X) \rightarrow \B  G(X)^{F_{X}-triv}$. Now from \S \ref{sub:Atiyah connections on trivialized bundles} and \S \ref{sub:Torsors trivialized along the absolute Frobenius} one get a natural functor $ \B G^F(X) \rightarrow \At^{triv}(G^{(p)},X) $.

\subsection{Torsors under Frobenius kernels and Maurer-Cartan map}%
\label{sub:Torsors under Frobenius kernels and Maurer-Cartan map}

This functor $\B G^F(X) \to  \At^{triv}(G^{(p)},X)$
is strongly related to the Maurer-Cartan map in the following sense. For $g\in G^{(p)}(X)$, we define $\partial(g)=X\times_{g,G^{(p)}} G$, seen as a $G^F$-torsor on $X$.

\begin{lemma}
\label{lem:frobenius_descent_maurer_cartan}
The following diagram is $2$-commutative :

\[ 
\begin{tikzcd}
\B G^F(X)	\ar[r] & \At^{triv}(G^{(p)},X)\\
	G^{(p)}(X)\ar[r,"\dd \log"]\ar[u,"\partial"] & \h^0(X,\Omega^1_{X/k}\otimes_{k}\Lie(G^{(p)}))  \ar[u,"s_e+\cdot "']
\end{tikzcd}
\]

\end{lemma}

\begin{proof}

It is enough to show that the following diagram is commutative :

\[ 
\begin{tikzcd}
	\B G^F(X)	\ar[r] & \At^{triv}(G^{(p)},X) \ar[d]\\
	G^{(p)}(X)\ar[r,"\dd \log"]\ar[u,"\partial"] & \h^0(X,\Omega^1_{X/k}\otimes_{k}\Lie(G^{(p)}))  \end{tikzcd}
\]
where the right vertical functor is the inverse equivalence of \eqref{eq:discrete-equivalence}. As this is just a matter of unraveling the definitions, we skip the proof.

\end{proof}

By construction, the functor $\B G^F(X) \to  \At^{triv}(G^{(p)},X)$ factors through $\At^{triv}(G^{(p)},X)^{C=\psi=0}$. Moreover:

\begin{lemma}
\label{lem:frobenius_descent_infinitesimal}
If $k$ is perfect, the functor $\B G^F(X) \to \At^{triv}(G^{(p)},X)^{C=\psi=0}$ is an equivalence of categories.
\end{lemma}

\begin{proof}
This has already been observed in \S \ref{sub:Torsors trivialized along the absolute Frobenius}.
\end{proof}

\begin{remark}
\label{rem:rigidity}
In particular,   $\B G^F(X)$ is rigid, which is  also  straightforward to show directly, since $X/k$ is smooth and $G^F$ is an infinitesimal group scheme over $k$.
\end{remark}

\begin{theorem}
\label{thm:torsors_Frobenius_kernels}
Assume that $k$ is perfect. There is a natural one to one map $\h^1(X,G^F)\to  \h^0(X,\Omega^1_{X/k}\otimes_k \Lie (G))^{C=\psi=0}$ that for each $g\in G(X)$ sends the image $\partial(g)$ of $g$ via the cobord map $\partial : G^{(p)}(X)\to \h^1(X,G^F)$ to $\dd \log(g)$, commutes with étale base change, and is characterized by these properties.
\end{theorem}

\begin{proof}

	According to equation \eqref{eq:discrete-equivalence}, Lemmas \ref{lem:frobenius_descent_maurer_cartan} and \ref{lem:frobenius_descent_infinitesimal} there is a natural $2$-commutative diagram:

\[ 
\begin{tikzcd}
	\B G^F(X)	\ar[r,"\sim"] & \At^{triv}(G^{(p)},X)^{C=\psi=0}\\
	G^{(p)}(X)\ar[r,"\dd \log"]\ar[u,"\partial"] & \h^0(X,\Omega^1_{X/k}\otimes_{k}\Lie(G^{(p)}))^{C=\psi=0} \isoarrow{u,"s_e+\cdot "'}
\end{tikzcd}
\]

All categories are rigid, so we can identify $\B G^F(X)$ with $\h^1(X,G^F)$ and because the functor $s_e+\cdot$ admits an explicit inverse, we get a ($1$-)commutative diagram :

\[ 
\begin{tikzcd}
\B G^F(X)	\ar[rd,"\alpha"] &\\
	G^{(p)}(X)\ar[r,"\dd \log"]\ar[u,"\partial"] & \h^0(X,\Omega^1_{X/k}\otimes_{k}\Lie(G^{(p)}))^{C=\psi=0} \end{tikzcd}
\]
where the diagonal map $\alpha$ is one to one. The fact that the formation of $\alpha$ commutes with \'etale base change on $X$ is clear, and the unicity follows, since any element of $\B G^F(X)$ is \'etale locally in the image of $\partial$ (by Corollary \ref{cor:image_MC_torsors_vanishing_curvatures}, or an easy direct proof).

\end{proof}

\begin{remark}
\label{rem:rigidity-key}
\begin{enumerate}
	\item The rigidity of $\B G^F(X)$ is the key fact that enables, starting from $\Omega \in \h^0(X,\Omega^1_{X/k}\otimes_{k}\Lie(G^{(p)}))^{C=\psi=0}$, corresponding on an \'etale cover $(X_i\to X)_{i\in I}$ to elements $g_i\in G^{(p)}(X_i)$, to glue together the local torsors $T_i=\partial(g_i)$. Since the $T_i$'s are given by explicit equations, this gives a very concrete description of the map $\alpha^{-1}$.  
\item The assumption that $k$ is perfect is essential. Namely, in the simple case $X=\mathbb A^1$, $G=\mathbb G_m$, it is easy to see that $\h^1(X,G^F)\simeq k^*/{{k^*}}^p$ whereas  $\h^0(X,\Omega^1_{X/k}\otimes_k \Lie (G))^{C=\psi=0}$ is trivial.
\item To describe explicitly the set $\h^0(X,\Omega^1_{X/k}\otimes_k \Lie (G))^{C=\psi=0}$ is probably a difficult question, as there is for a general, non-commutative group $G$, no explicit formula (at least known to us) for the $p$-curvature. Interestingly, when $G$ is commutative, one can express the $p$-curvature $\psi$ using only the Cartier operator $\C$ and the $p$-th power operation on $\Lie(G)$, see Appendix \ref{sec:p-curvature_abelian_structure_groups}.

\end{enumerate}
 
\end{remark}

\appendix

\section{Classical Frobenius descent}
\label{sec:classical_frobenius_descent}

In this paragraph, we put $S=\spec k$. We recall Cartier's celebrated result on Frobenius descent. This is nowadays seen as fppf descent along the relative Frobenius morphism $F_{X/S}$ of $X/S$, that is the unique morphism making the following diagram commute :

\begin{equation}
\label{frobenius_diagram}
\begin{tikzcd}
	X \arrow[bend right]{ddr}{}  \arrow[dotted]{dr}{F_{X/S}} \arrow[bend left]{drr}{F_X} & & \\
 & X^{(p)} \arrow{d}{} \arrow{r}{\pi_{X/S}} & X \arrow{d}{} \\
 & S \arrow{r}{F_S} & S
\end{tikzcd}
\end{equation}
where $F_S$ (respectively $F_X$) is the absolute Frobenius of $S$ (respectively of $X$).

Let $\mathcal F$ be a quasi-coherent sheaf on $X^{(p)}$. Then $F_{X/S}^*\mathcal F= \mathcal{O}_X\otimes _{\mathcal{O}_{X^{(p)}}}\mathcal F$ is endowed with a canonical connection  $\nabla^{can}$ given by  $f\otimes s\mapsto (1\otimes s)\otimes\mathrm{d}f$. This connection is integrable with vanishing $p$-curvature.

In the other direction, if $(\mathcal E,\nabla)$ is a quasi-coherent sheaf endowed with an integrable connection $X$, of vanishing $p$-curvature, then the sub-sheaf of horizontal sections $\mathcal E^\nabla=\left\{ s\in \mathcal E / \nabla(s)=0 \right\}$ is a quasi-coherent sheaf on $X^{(p)}$.

\begin{theorem}[Cartier]
\label{thm:cartier}
Assume $X/S$ is smooth.

The functors $\mathcal F \mapsto (F_{X/S}^*\mathcal F,\nabla_{can})$ and $(\mathcal E,\nabla)\mapsto \mathcal E^\nabla$ are inverse equivalences between the category of quasi-coherent sheaves on $X^{(p)}$ and the category of quasi-coherent sheaves on $X$ endowed with an integrable connection of vanishing $p$-curvature.
\end{theorem}

\begin{proof}
	See \cite[Theorem (5.1)]{katz:nilpotent}.
\end{proof}

\section{Frobenius pull-back of torsors}
\label{sec:frobenius_pull_back_of_torsors}

We use notations of diagram \eqref{frobenius_diagram}.

\begin{lemma}
\label{lem:frobenius_pull_back_of_torsors}
Let $G/S$ be an affine group scheme and $P\to X$ be a $G$-torsor. There exists a natural isomorphism of $G^{(p)}$-torsors:
\[ P \wedge^G G^{(p)} \simeq F_X^* P \; \; .\] 
\end{lemma}

\begin{proof}
	Notice that $F_X^* P = X\times_{F_X,X} P=P^{(p/X)}$, moreover the relative Frobenius morphism $F_{P/X}:P\to P^{(p/X)}=F_X^* P$ is equivariant with respect to the morphism $F_{G}:G\to G^{(p)}$, hence give rise to an (iso)morphism of $G^{(p)}$-torsors $P \wedge^G G^{(p)} \simeq F_X^* P$.
\end{proof}

\section{\texorpdfstring{$p$}{p}-curvature for abelian structure groups}
\label{sec:p-curvature_abelian_structure_groups}

In this section, we assume that $ \car k \neq 2$.

Let $G/k$ be an affine, smooth, abelian group scheme. In this situation, there is an explicit description of the $p$-curvature map in terms of the Cartier operator $\C: {F_{X/S}}_*(\Z(\Omega^1_{X/k})) \to \Omega^1_{X^{(p)}/k}$.

Let us first remark that if $\Omega \in \h^0(X,\Omega^1_{X/k}\otimes_k \Lie (G))$, since $G$ is abelian, according to Lemma \ref{lem:curvature-GL}, we have that $C_{s+\Omega}=\mathrm d \Omega$. Hence the source of the $p$-curvature map can be written as 
$\h^0(X,\Z(\Omega^1_{X/k}))\otimes_k \Lie (G)$.
\begin{proposition}
\label{prop:hitchin_map_abelian}

If $G/k$ be an affine, smooth, abelian group scheme, the $p$-curvature morphism :
$$\psi: \h^0(X,\Z(\Omega^1_{X/k}))\otimes_k \Lie (G) \to
\h^0(X^{(p)},\Omega^1_{X^{(p)}/k})\otimes_k \Lie (G)$$
equals $\pi_{X/k}^*\otimes \cdot^p- \C\otimes \id$, where $\pi_{X/k}: X^{(p)} \to X$ is the canonical morphism, $\cdot^p$ is the $p$-th power map on $\Lie (G)$, and $\C$ is the Cartier operator.
\end{proposition}

\begin{proof}
	This is	\cite[(A.8)]{chen-zhu:non_abelian} (but see also \cite[Lemma (2.2)]{artin-milne:duality}).
\end{proof}

\printbibliography

\end{document}